\documentclass[12pt]{amsart}  

\usepackage{comment}

\usepackage{amsmath}
\usepackage{latexsym,amssymb}
\usepackage[mathscr]{eucal}

\usepackage{enumitem}

\usepackage{hyperref}  

\newtheorem{thm}{Theorem}[section]
	
\newtheorem{pro}[thm]{Proposition}
\newtheorem{cor}[thm]{Corollary}
\newtheorem{dfn}[thm]{Definition}
\newtheorem{result}[thm]{Result}

\theoremstyle{definition}

\newtheorem{eg}[thm]{Example}

\newcommand{\bea}{\begin{eqnarray*}}
\newcommand{\eea}{\end{eqnarray*}}

\newcommand{\ben}{\begin{enumerate}}
\newcommand{\een}{\end{enumerate}}

\newcommand{\bi}{\begin{itemize}}
\newcommand{\ei}{\end{itemize}}

\newcommand{\mc}{\mathcal}

\newcommand{\D}{\mathcal D}

\newcommand{\Pcon}{\mathcal P}
\newcommand{\Q}{\mathcal Q}
\newcommand{\C}{\mathcal C}
\newcommand{\I}{\mathcal I}

\DeclareMathOperator{\Dom}{Dom}
\DeclareMathOperator{\Img}{Im}

\DeclareMathOperator{\Reg}{Reg}

\begin{document}

\title{D-inverse constellations}
\author{Victoria Gould and Tim Stokes}

\date{}

\begin{abstract} Constellations are partial algebras in the sense that they possess a partial product, and a unary operation modelling domain.  They were first used to give an ESN-style theorem for left restriction semigroups in terms of so-called inductive constellations.  Here, we consider constellations in which elements have a suitable notion of inverse, giving the notion of a D-inverse constellation.  We show that there is a categorical isomorphism between the category of ordered groupoids and the category of D-inverse constellations.  This may be viewed as a generalisation of the ESN theorem, which relates the category of inductive groupoids to the category of inverse semigroups.
\end{abstract}

\maketitle

\section{Introduction}\label{sec:intro}  

The Ehresmann-Schein-Nambooripad (ESN) Theorem establishes a categorical equivalence, indeed an isomorphism, between the category of inverse semigroups and the category of inductive groupoids;  we refer the reader to \cite{lawsoninverse} for an exposition of this result and its history. One can think of the ESN Theorem as replacing an algebraic structure by an ordered one, which in this case is a small category 
equipped with an order that allows one to reconstruct the `lost' algebraic information. 
Following the establishment of the ESN theorem, Lawson provided a correspondence along similar lines in \cite{lawson1} that connected so-called Ehresmann semigroups (bi-unary semigroups with both and domain and range-like operations) to certain types of ordered categories.  Two-sided restriction semigroups are Ehresmann semigroups, and are covered as a special case in \cite{lawson1}.  However, one-sided restriction semigroups are perhaps even more natural than two-sided ones since they provide the algebraic models for partial functions, so there was interest in obtaining an analogous result for them.  The problem was the lack of a range operation, and more generally, the lack of sufficient left-right dual properties. 

Constellations are ``one-sided" versions of categories, in which there is a notion of domain but in general no notion of range.  They were first defined by Gould and Hollings in \cite{constell}, where the purpose was to obtain a variant of the ESN  Theorem for  left restriction semigroups, using so-called inductive constellations, where order again plays a major role. Constellations have since been studied further for their own sake in \cite{constgen,constrange} and their connections with categories explored. 

An inductive groupoid is an ordered groupoid  in which the domain elements are a meet semilattice under the given order.  The parent category of ordered groupoids shares many of the nice properties of inductive groupoids and hence of inverse semigroups:   indeed 
ordered groupoids are important in the  study of inductive groupoids.  There is therefore interest in whether ordered groupoids in general correspond to any kind of purely algebraic structure. The aim of this work is to show that they do - indeed they correspond exactly to constellations having a natural notion of inverse. 

In \cite{csg}, the notion of a ``true inverse'' of an element in a left restriction semigroup $S$ was considered: $s\in S$ has {\em true inverse} $t$ if $st=D(s)$ and $ts=D(t)$.  It was shown that true inverses are unique if they exist, and that if every element has a true inverse then $S$ is an inverse semigroup in which $s'$ (that is, the inverse of $s$ in the  inverse semigroup) is the true inverse of $s$.  Conversely, it is easy to show that every inverse semigroup gives rise to a left restriction semigroup in which every element has a true inverse, by setting $D(s)=ss'$ for all $s\in S$. These constructions are easily seen to be mutually inverse, so one may view inverse semigroups as nothing but left restriction semigroups in which every element has a true inverse.  

Our main result, Theorem~\ref{invcorresp}, uses an analogous notion of D-inverse  to define analogues of ordered groupoids within the class of constellations.  This reframing of an ordered structure to a purely algebraic one gives a new viewpoint, and allows alternative approaches to the theory of ordered groups and of inverse semigroups, already demonstrated in \cite{kaden}.

In Section \ref{sec:back}, we recall the required background notions.  
In Section \ref{sec:Dreg}, attention turns to the definitions and basic properties of D-regular and D-inverse constellations, defined by analogy with left restriction semigroups in which each element has a true inverse.  We show that small D-inverse constellations have a Cayley-style representation theorem in terms of one-to-one partial maps, generalising the Vagner-Preston theorem for inverse semigroups.

In Section \ref{sec:equiv}, we show in Theorem~\ref{invcorresp} that D-inverse constellations are nothing but ordered groupoids ``in disguise".  We do this by exploiting a known correspondence between constellations with range and ordered categories with restriction, established in \cite{constrange}.  This correspondence raises the possibility of using D-inverse constellations as a tool to study ordered groupoids and hence invere semigroups.  

 Indeed, in Section \ref{sec:construct}, we give an illustration of how one might use D-inverse constellations, by giving a construction of D-inverse constellations from semigroups with a distinguished set of idempotents.  This construction generalises both Nambooripad's construction of an ordered groupoid from a regular semigroup and Lawson's construction of an ordered groupoid from the partial isometries of a semigroup with involution.

Inverse semigroups are not defined using a domain operation, although such an operation can be defined in any inverse semigroup, in terms of composition and inverse.  So it is not surprising that D-inverse constellations may also be given a ``domain-free" description, using the notion of a pre-constellation (these satisfy the defining laws of constellations that do not involve the domain operation).  This is the subject of Section \ref{sec:pre}, where we show that there is a more general notion of ``inverse pre-constellation" amongst which D-inverse constellations may be characterised in purely pre-constellation terms.

\section{Background} \label{sec:back}   

We set out the necessary notions for this article, and clarify our notational conventions.

\subsection{Notational conventions}

Throughout, we generally write functions on the right of their arguments rather than the left, so ``$xf$" rather than ``$f(x)$".  Correspondingly, we write function compositions left to right, so that ``$fg$" is ``first $f$, then $g$".  An exception to this is unary operation application: if $D$ is a unary operation on the set $S$, we write $D(s)$ for $s\in S$.

\subsection{Partial binary algebras}

In what follows our convention is that if an algebra or partial structure is denoted by $\mathcal{A}$ (or $\mathscr{A}$), then we use $A$ for the underlying set.

\begin{dfn}
Let $\C=(C,\, \cdot\,)$ be a class $C$ $($often a set$)$ equipped with a partial binary operation $\cdot$.  
\bi 
\item We say $e\in C$ is a \emph{right identity} if it is such that, for all $x\in C$, if $x\cdot e$ exists then it equals $x$; left identities are defined dually.
\item The collection of right identities of $\C$ is denoted by $RI(\C)$.
\item An \emph{identity} is both a left and right identity. 
\item  An element $e\in C$ is {\em idempotent} if $e\cdot e$ exists and equals $e$.
\item We let $E(\C)=\{e\in C\mid e\cdot e=e\}$. 
\ei 
\end{dfn}

\subsection{Categories and groupoids}

For the purposes of this paper it is convenient to give  an object-free formulation of a category, as occurs in the 
other works in this area, for example \cite{constgen} and \cite{lawson1},  which have brought us to this point in the research development. Recall that in this framework a \emph{category} $\C=(C,\, \circ\,)$ is a class $C$ equipped with a partial binary operation $\circ$ satisfying the following:
\bi
\item[] (Cat1) $x\circ(y\circ z)$ exists if and only if $(x\circ y)\circ z$ exists, and then the two are equal;
\item[] (Cat2) if $x\circ y$ and $y\circ z$ exist then so does $x\circ(y\circ z)$;
\item[] (Cat3) for each $x\in \C$, there are identities $e,f$ such that $e\circ x$ and $x\circ f$ exist.
\ei

The identities $e,f$ in (Cat3) are easily seen to be unique,  
and  we write $D(x)=e$ and $R(x)=f$.  (Note that this is the opposite of the convention often used, and corresponds to the fact that we view a composition of functions $fg$ as ``first $f$, then $g$", discussed earlier.)

It also follows easily that every identity $e$ is idempotent, and $D(e)=e=R(e)$, and that the collection of domain elements $D(x)$ (equivalently, range elements $R(x)$) is precisely the collection of identities in the category; we denote this set by $D(\C)$.    If $x,y$ are elements of a category $\C$, then the product $x\circ y$ exists if and only if $R(x)=D(y)$, and if $x\circ y$ exists then $D(x\circ y)=D(x)$ and $R(x\circ y)=R(y)$.  

A {\em functor} $\rho$ from $\C$ to $ \D$, where $\C=(C,\, \circ\,)$
and $\D=(D,\, \circ\,)$ are categories, is a map $\rho:C\rightarrow D$ such that for all $x\in \C$, $D(x)\rho=D(x\rho)$, $R(x)\rho=R(x\rho)$, and if $x\circ y$ exists in $\C$, then it equals $x\rho\circ y\rho$ (which must exist); that this is equivalent to the usual object-morphism notion of functor follows from the definitions.

In view of this, particularly when our category is small, we will often view a category as a partial algebra $\C=(C,\, \circ\, ,D,R)$.  With this viewpoint, a small category with a single identity (equivalently, in which all products are defined) is nothing but a monoid.

A {\em groupoid} is a structure $\Q=(Q,\, \circ\, ,\, '\, ,D,R)$ such that $(Q,\, \circ\, ,D,R)$ is a category and $'$  is a unary operation  on $Q$, such that for all $s\in Q$, $s\circ s'=D(s)$ and $s'\circ s=R(s)$.  Again with this viewpoint, a small groupoid with a single identity is simply a group.

\subsection{Ordered categories and groupoids}

Following \cite{lawson1} (and using nomenclature consistent with that used there), we say a category $\C=(C,\, \circ ,D,R)$ equipped with a binary relation $\leq$ on $C$ is an {\em ordered category with restrictions} if:
\ben[label=\textup{(OC\arabic*)}]
\item $(C,\leq)$ is a poset;
\item $x\leq y$ implies $D(x)\leq D(y)$ and $R(x)\leq R(y)$;  \label{DRmono}
\item  if $x_1\leq x_2$ and $y_1\leq y_2$ and both $x_1\circ y_1$ and $x_2\circ y_2$ exist,  then $x_1\circ y_1\leq x_2\circ y_2$;\label{OC3}
\item  if $D(x)=D(y)$ and $R(x)=R(y)$ and $x\leq y$ then $x=y$ ($\leq$ is trivial on hom-sets). \label{hom}
\een
It has
\ben[label=\textup{(OC8(i))}]
\item {\em restrictions} if, for all $x\in C$ and $e\in D(\C)$ for which $e\leq D(x)$, there exists a unique element, called the {\em restriction of $x$ to $e$} and denoted $e|x$, such that $e|x\leq x$ and $D(e|x)=e$,  \label{OC8i}
\een

(We remark that the term `ordered category' is used in the literature to mean a category satisfying the conditions above, together with the left-right dual of (OC8(i)).)

An {\em ordered groupoid} $\Q=(Q,\, \circ\, ,\, '\, ,D,R,\leq)$ is a groupoid $\Q=(Q,\, \circ\, ,\, '\, ,D,R)$ equipped with a partial order $\leq$ obeying the following axioms:

\ben[label=\textup{(OG\arabic*)}]
\item $x\leq y$ implies $x'\leq y'$;

\item Law \ref{OC3} for ordered categories with restrictions:  

\item Law \ref{OC8i} for ordered categories with restrictions.  \label{OG3}
\een
Note that although Laws \ref{DRmono} and \ref{hom} in the definition of an ordered category with restrictions given earlier are not part of the definition of an ordered groupoid, they follow easily.  So ordered groupoids are certainly ordered categories with restriction. It follows that there is also a notion of corestriction, defined as $s|e=(e|s')'$ whenever $e\leq R(s)$, satisfying the obvious axiom dual to \ref{OG3}, and this is often included as part of the definition.

\subsection{Inductive groupoids and inverse semigroups} 

To provide the context for our main result, Theorem~\ref{invcorresp}, we recall the ESN theorem. 

First, an {\em inductive groupoid} is an ordered groupoid $\Q$ in which $D(\Q)$ forms a semilattice under the ordering. On the other hand, a  semigroup $S$ is {\em inverse} if for every $a\in S$ there exists a {\em unique} $b\in S$ such that $a=aba$ and $b=bab$; usually, 
$b$ is denoted by $a'$. In an inverse semigroup the idempotents of $S$ commute, and thus form a semilattice under the ordering
$a\leq b$ if $ab=a$, where meets are given by the semigroup product. (For the details, see \cite{lawsoninverse}.)  The canonical example of an inverse semigroup is the symmetric inverse monoid $\mathscr{I}_X$ on a set $X$. The elements 
of $\mathscr{I}_X$ are partial bijections of $X$ and the composition $fg$  of $f$ and $g$ in
$\mathscr{I}_X$ is that of  partial functions.

The ESN theorem says that the category of inductive groupoids with appropriately defined functors is equivalent (in fact, isomorphic) to the category of inverse semigroups and semigroup morphisms. An exposition of this result and its history is given in \cite{lawsoninverse}. One aspect crucial to setting up this equivalence is the inductive nature of the groupoid. We aim to answer the question: what algebraic construct corresponds to an ordered groupoid? For this, we use constellations.

\subsection{Constellations}

In what follows we make use of an earlier equivalent definition given in \cite{constell} and  make use of a result in \cite{constgen}. We define a \emph{constellation} $(\Q,\cdot)$  to be a class $\Q$ equipped with a partial binary operation $\cdot$ satisfying the following:
\bi
\item[] (Const1) if $x\cdot(y\cdot z)$ exists then $(x\cdot y)\cdot z$ exists, and then the two are equal;
\item[] (Const2) if $x\cdot y$ and $y\cdot z$ exist then so does $x\cdot(y\cdot z)$;
\item[] (Const3) for each $x\in P$, there is a unique right identity $e$ such that $e\cdot x=x$.
\ei
Since $e$ in (Const3) is unique given $x\in \Q$, we call it $D(x)$.  
It follows that $D(\Q)=\{D(s)\mid s\in \Q\}=RI(\Q),$ the set of right identities of $\Q$.  We call $D(\Q)$ the {\em projections} of $\Q$.  We adopt the usual convention of referring to the constellation $(\Q,\cdot)$ simply as $\Q$ if there is no ambiguity. However, following the convention for categories, and since $D$ may be viewed as a unary operation, we also often view constellations as partial algebras $(\Q,\cdot\, ,D)$.      We say a constellation $\Q$ is {\em small} if $\Q$ is a set.

The following are some useful basic facts about constellations, to be found in \cite{constell} or \cite{constgen}.

\begin{result}  \label{2p3}
For elements $s,t$ of the constellation $\Q$, we have that $s\cdot t$ exists if and only if $s\cdot D(t)$ exists, and then $D(s\cdot t)=D(s)$.
\end{result}

{As shown in \cite{constgen}, every category becomes a constellation when the operation $R$ is ignored. It is easy to see that a constellation arises from a category as a reduct in this way if and only if for all $s\in \Q$ there is a unique $e\in D(\Q)$ such that $s\cdot e$ exists, and then $R(s)=e$ when $\Q$ is viewed as a category.

We emphasise that by very definition a constellation $\Q$ is a (partial) algebra. However, as for inverse and left restriction semigroups, from an algebraic structure we obtain an ordered structure. Specifically, we define the relation 
 $s\leq t$ for $s,t\in \Q$ whenever $s=e\cdot t$ for some $e\in D(\Q)$ (equivalently, $s=D(s)\cdot t$). The restriction to $D(\Q)$, also denoted by $\leq$,  is then given by $e\leq f$ if and only if $e\cdot f(=e)$ exists. 
A constellation $\Q$ is {\em normal} if $\mbox{for all }e,f\in D(\Q),\mbox{ if $e\cdot f$ and $f\cdot e$ exist, then }e=f$.   This is a property reminiscent of idempotents in a semigroup commuting, and relates to whether the given relations on $\Q$ are  partial orders.

\begin{result}  \label{parnormal} Let
$\Q$ be a constellation.  Then the relation $\leq$ on $\Q$ is a quasiorder which we call the {\em natural quasiorder}; its restriction to  $D(\Q)$ is called the {\em standard quasiorder}.   In both these cases, the quasiorder is a partial order if and only if $\Q$ is normal, and then we use ``order" rather than ``quasiorder".
\end{result}

Certainly not all constellations are normal. Indeed, any quasi-ordered set $(P,\leq)$  carries the structure of a constellation 
$\mathcal{P}=(P, \cdot)$, where  $x\cdot y$ is defined if and only if $x\leq y$ and then $x\cdot y=x$ \cite{constell}; this constellation is normal if and only if $(P,\leq)$ is a poset. An important example of a small normal constellation, introduced in \cite{constell}, is ${\mc C}_X$, with underlying set consisting of the partial functions on the set $X$, in which $s\cdot t$ is the usual composite of $s$ followed by $t$ provided $\Img(s)\subseteq \Dom(t)$, and undefined otherwise, and $D(s)$ is the restriction of the identity map on $X$ to $\Dom(s)$. 

A {\em subconstellation} $\Pcon$ of a constellation $\Q$ is a subset of $\Q$ that is closed under the constellation product wherever it is defined, and closed under $D$; then $\Pcon$ is  a constellation in its own right as is easily seen \cite{constgen}.
An important subconstellation of  $\C_X$ consists of the elements which are one-one as (partial) maps; this yields the  constellation $\I_X$, which has the same underlying set $I_X$ as $\mathscr{I}_X$, but a restricted set of products.

For constellations, the notion of morphism is as follows.  If $\Q_1,\Q_2$ are constellations, a function $\rho:\Q_1\rightarrow \Q_2$ is a {\em radiant} if for all $s,t\in \Q_1$ for which $s\cdot t$ exists, then so does $(s\rho)\cdot (t\rho)$ and  $(s\cdot t)\rho=(s\rho)\cdot (t\rho)$, and $D(s\rho)=D(s)\rho$.  As observed in \cite{constell}, in a standard way we have a (large) category with objects constellations and the  morphisms being radiants.  

A {\em strong} radiant $\rho:\Q_1\rightarrow \Q_2$ is a radiant for which, for all $s,t\in \Q_1$, $s\cdot t$ exists if and only if $s\rho\cdot t\rho$ does.  We say $\rho:\Q_1\rightarrow \Q_2$ is an {\em isomorphism} if $\rho$ is a strong bijective radiant, and an {\em embedding} if it is a strong injective radiant. 
Note that an isomorphism is precisely a  radiant with a radiant inverse.  If $\rho:\Q_1\rightarrow \Q_2$ is an embedding then the image of $\rho$ is a subconstellation of $\Q_2$ that is isomorphic to $\Q_1$.

 We say a constellation $\Q_1$ {\em embeds} in the constellation $\Q_2$ if there is an embedding $\rho:\Q_1\rightarrow \Q_2$.   It was shown in \cite{constgen} that every small normal constellation embeds as a subconstellation in the (normal) constellation $\C_X$ for some choice of $X$.

\section{D-regular and D-inverse constellations}  \label{sec:Dreg}

In this section we introduce the central new notion of this article, namely that of a D-inverse constellation.

\subsection{D-regularity} \label{sec:regular}

Recall that a semigroup $S$ is {\em regular} if for all $a\in S$ there exists $b\in S$ such that $a=aba$; the element $ab$ is then an idempotent and a left identity {\em for $a$.} If a regular semigroup is inverse then clearly it is regular and we have pointed out that its idempotents commute; in fact the converse is also true  \cite{lawsoninverse}.  We here consider a suitable variant of this notion for constellations.

\begin{dfn}
The constellation $\Q=(Q,\, \cdot\, , D)$ is {\em D-regular} if for all $a\in Q$ there is $b\in Q$ for which $a\cdot b=D(a)$. 
\end{dfn}

In general, for a constellation $\Pcon$, the sets $RI(\Pcon)$ and $E(\Pcon)$ can be different, as may easily be seen by considering $\C_X$, although we always have $D(\Q)=RI(\Pcon)\subseteq E(\Pcon)$.   

\begin{pro}  \label{EPinRIP}
In a D-regular constellation $\Pcon$, we have that $E(\Pcon)\subseteq RI(\Pcon)$ and so the two sets are equal $($and hence both equal $D(\Pcon)$$)$.
\end{pro} 
\begin{proof}
Pick $e\in E(P)$, and suppose that $e'$ is such that $e\cdot e'=D(e)$.  Then $D(e)=e\cdot e'=(e\cdot e)\cdot e'=e\cdot(e\cdot e')=e\cdot D(e)=e$, so $e\in D(P)$.
\end{proof}

A property related to D-regularity is the following. At first sight the connection seems suprising, but the reader should recall that in a constellation we have only a partial binary operation, the domain of which is tightly controlled by the projections.

\begin{dfn}
The  constellation $\Pcon$ is {\em right cancellative} if $a\cdot c=b\cdot c$ implies that  $a=b$, for any $a,b,c\in P$.
\end{dfn}

This definition agrees with the category definition if $\Pcon$ is a category, and is relevant here because of the following.

\begin{pro}  \label{monic}
Every D-regular constellation is right cancellative.  
\end{pro}
\begin{proof}
Suppose that $\Q=(Q,\, \cdot\, , D)$ is a D-regular constellation.  If $a,b,c\in Q$ and $a\cdot c=b\cdot c$ then since there exists $d\in Q$ for which $c\cdot d=D(c)$, both $a\cdot (c\cdot d)$ and $b\cdot (c\cdot d)$ exist.
Then  $a\cdot (c\cdot d)=(a\cdot c)\cdot d=(b\cdot c)\cdot d=b\cdot (c\cdot d)$, so $a\cdot D(c)=b\cdot D(c)$, giving that  $a=b$.
\end{proof}

Let $X$ be a set. The subconstellation of one-to-one partial functions $\I_X=(I_X,\, \cdot\, ,D)$ of $\C_X$ is right cancellative and normal (recalling that $s\cdot t$ is interpreted as ``first $s$ then $t$'').  Conversely, there is a ``Cayley theorem" for right cancellative constellations, as follows.

\begin{pro}  \label{rcinj} A small constellation $\Pcon=(P,\, \cdot\, ,D)$  is normal and right cancellative if and only if it embeds in $\I_P=(I_P,\, \cdot\, ,D)$. 
\end{pro}
\begin{proof}  One direction is clear, since the properties of being normal and of being right cancellative are inherited by subconstellations.

For the converse, suppose that $\Pcon$ is normal and right cancellative.
It was shown in \cite{constgen} that the radiant $\rho:\Pcon\rightarrow \C_P$ taking $s\in P$ to $\rho_s\in C_P$ given by $x\rho_s:=x\cdot s$ for all $x\in P$ for which the latter is defined, is an embedding of $\Pcon$ in $\C_P$ as a constellation.  It is easy to see that each $\rho_s$ is injective if (and only if) $\Pcon$ is right cancellative, and so the image of 
 $\rho$ lies within $\I_P$.
\end{proof}

\begin{pro}  \label{Dregim}
Suppose that $\Pcon,\Q$ are constellations and $\Pcon$ is D-regular, and that $\rho: \Pcon\rightarrow \Q$ is a surjective radiant. Then $\Q$ is D-regular.
\end{pro}
\begin{proof}
Suppose that $s\cdot t=D(s)$, where $s,t\in P$.  Then $D(s\rho)=D(s)\rho=(s\cdot t)\rho=s\rho\cdot t\rho$, so, since $\rho$ is surjective, 
we deduce that $\Q$ is D-regular.
\end{proof} 

\subsection{D-inverses}\label{sub:dinv}

\begin{dfn}
If $\Q$ is a constellation, we say the element $s\in Q$ has {\em D-inverse} $t\in Q$ if $s\cdot t=D(s)$ and $t\cdot s=D(t)$.  
\end{dfn}  

If $t$ is a D-inverse for $s$ in a constellation, then obviously $s$ is a D-inverse for $t$ also. 

We have the following generalisation of a familiar fact about inverses in categories (which are normal when viewed as constellations).

\begin{pro}  \label{normunique}
In a constellation $\Q$, every $e\in D(\Q)$ is a D-inverse for itself, and each element of $\Q$ has at most one D-inverse if and only if $\Q$ is normal.
\end{pro}
\begin{proof}  
For $e\in D(\Q)$, $e\cdot e=e=D(e)$ so $e$ is a D-inverse of itself.

Suppose that $\Q$ is normal and $s_1$ and $s_2$ are both D-inverses of $s\in Q$.  Then $s_1\cdot s$ and $s\cdot s_1$ exist, whence so does $s_1\cdot(s\cdot s_1)=s_1\cdot D(s)=s_1$ by (Const2).  By (Const1), 
$$s_1 =s_1\cdot D(s)=s_1\cdot (s\cdot s_2)=(s_1\cdot s)\cdot s_2=D(s_1)\cdot s_2,$$
so $s_1\leq s_2$ under the natural order on $\Q$.  Similarly, $s_2\leq s_1$, and so $s_1=s_2$ by Result \ref{parnormal}.

Conversely, if D-inverses are unique, suppose that $e\cdot f$ and $f\cdot e$ both exist where $e,f\in D(\Q)$.  Then $e\cdot f=e=D(e)$ and $f\cdot e=f=D(f)$, and so $f$ is a D-inverse of $e$, whence $f=e$ by uniqueness.  This establishes normality.
\end{proof}

\begin{dfn}
We say the constellation $\Q$ is a {\em D-inverse constellation} if every element of $\Q$ has a unique D-inverse.  We use $s'$ to denote the unique D-inverse of $s$. 
\end{dfn}  

From Proposition \ref{normunique}, we obtain the following.

\begin{cor}  \label{corDinv}
The constellation $\Q$ is D-inverse if and only if every element of $\Q$ has a D-inverse and $\Q$ is normal.
\end{cor}

It is clear from the definitions that if  $\Pcon$ and $\Q$ are D-inverse constellations with $\rho: \Pcon\rightarrow \Q$ a radiant, then $(s\rho)'=s'\rho$. We can say something more.

\begin{pro}  \label{Dregim2}
Suppose that $\Pcon,\Q$ are constellations, with $\Pcon$ being D-inverse and $\Q$ normal, and suppose that  $\rho: \Pcon\rightarrow \Q$ is  a surjective radiant. Then $\Q$ is D-inverse, and $(s\rho)'=s'\rho$.
\end{pro}
\begin{proof}
From the proof of Proposition \ref{Dregim}, we see that $s'\rho$ is a D-inverse of $s\rho$. Since $\rho$ is surjective we deduce from  Corollary \ref{corDinv} that $\Q$ is D-inverse, with  $(s\rho)'=s'\rho$.
\end{proof}

A monoid is a D-inverse constellation if and only if it is a group.  Every poset $P$ gives a D-inverse constellation if we define $e\cdot f=e$ if and only if $e\leq f$; this is a constellation in which $D(e)=e$ for all $e\in P$, as noted first in \cite{constell}, and normal as noted in \cite{constgen}.  Since $e\cdot e=e=D(e)$ for all $e\in P$, we have that $e$ is a D-inverse of itself.

A further significant example of a D-inverse constellation is the constellation $\I_X$ of one-to-one partial functions on the set $X$. The constellation $\I_X$ comes }equipped with the additional operation of inversion: this is because for all $f\in \I_X$, if $f'$ is the inverse of $f$ in the usual partial function sense, then $f\cdot f'=D(f)$ and $f'\cdot f=D(f')$, so $f'$ is the D-inverse of $f$ in the constellation sense also. The  inverse semigroup $\mathscr{I}_X$ and the D-inverse constellation $\I_X$ have the same underlying set,  but  $f\cdot g$ exists in $\I_X$ if and only if $fD(g)=fgg'=f$ in $\mathscr{I}_X$.  In fact this generalises as follows, courtesy of the Vagner-Preston representation theorem, which says that any inverse semigroup embeds as a subsemigroup of some $\mathscr{I}_X$, (although a direct proof is straightforward).

\begin{pro} \label{inv2inv}
If $S$ is an inverse semigroup, then defining $s\cdot t=st$ exactly when $stt'=s$ makes $S$ into a D-inverse constellation.
\end{pro}

We finish this section with a representation theorem for D-inverse constellations. We say that a subconstellation of a D-inverse constellation is an {\em inverse subconstellation} if it is closed under taking of inverses. Clearly such a subconstellation is itself a D-inverse constellation.

\begin{thm}  \label{normiso}  Let $\Q=(Q,\cdot\, ,D)$ be a small constellation. Then
$\Q$ is D-inverse if and only if it is an inverse subconstellation of some $\I_Q$.
\end{thm}
\begin{proof}  One direction follows from comments above. 
Suppose now that $\Q$ is D-inverse, and consider the radiant $\rho:\Q\rightarrow \C_X$ as in the proof of Proposition \ref{rcinj}.  By that result, $\rho$ embeds the constellation $\Q$ into the consellation $\I_X$.  It remains to show that for $s\in Q$
we have $s'\rho=(s\rho)'$.  Now for $x\in\Dom(s\rho)$, we have that $x\rho_s=x\cdot s$, and of course $s\cdot s'$ exists. It follows that we have that $x\cdot (s\cdot s')=(x\cdot s)\cdot s'$ exists, so that $x\cdot s\in\Dom(s'\rho)$ and then that 
$\Dom(D(s\rho))=\Dom (s\rho \, s'\rho)$. Further, 
\[(x\cdot s)\cdot s'=x\cdot (s\cdot s')=x\cdot D(s)=x,\]
so that  $\rho_s\cdot\rho_{s'}=\rho_{D(s)}=D(s\rho)$.  Similarly, $\rho_{s'}\cdot\rho_{s}=D({s'}\rho)$.  Together, these statements show that $s'\rho=\rho_{s'}=(\rho_s)'=(s\rho)'$, as required.
\end{proof}

From (Const2), in a D-inverse constellation $\Q$ we have  $s\cdot (s'\cdot s)=(s\cdot s')\cdot s=s$ for all $s\in Q$.  Other familiar facts of inverse semigroup theory such as $(st)'=t's'$ do not carry over since the existence of $s\cdot t$ does not ensure that $t'\cdot s'$ exists (even if both $s,t$ have D-inverses).  For example, let $X=\{x,y\}$, with $s,t\in \I_X$ defined as follows: $s=\{(x,x)\}$ and $t=1$ (the identity function on $X$).  Then in $\I_X$, $s'=s$, $t'=t$, $s\cdot t=s$ yet $t'\cdot s'=t\cdot s$ does not exist.  However, we do have the following.

\begin{pro}  \label{sands}
If $\Q$ is a D-inverse constellation and $s,t\in Q$ are such that both $s\cdot t$ and $t'\cdot s'$ exist, then $(s\cdot t)'=t'\cdot s'$.
\end{pro}
\begin{proof}
Since $s\cdot t$, $t\cdot t'$ and $t'\cdot s'$ all exist, so does $t\cdot (t'\cdot s')=(t\cdot t')\cdot s'$ and hence $s\cdot (t\cdot (t'\cdot s'))=(s\cdot t)\cdot (t'\cdot s')$ which also equals $s\cdot((t\cdot t')\cdot s')=(s\cdot(t\cdot t'))\cdot s'=(s\cdot D(t))\cdot s'=s\cdot s'=D(s)=D(s\cdot t)$.  So $(s\cdot t)\cdot (t'\cdot s')=D(s\cdot t)$.  Similarly, on interchanging the role of $s,t'$ and $s',t$, we have $(t'\cdot s')\cdot (s\cdot t)=D(t'\cdot s')$.  We deduce that $(s\cdot t)'=t'\cdot s'$.
\end{proof}

In the small case, the above result also follows from the embeddabity of $\Q$ in $\I_X$ as in Theorem \ref{normiso}.

\section{D-inverse constellations are exactly ordered groupoids}  \label{sec:equiv}

In the examples of constellations considered in \cite{constgen}, most had a notion of range, satisfying some natural properties.  In \cite{constrange}, it was shown that they are nothing but ordered categories with restriction, a class containing all ordered groupoids.  In this section we show that every D-inverse constellation is a constellation with range, and that when viewed as an ordered category with restriction, it is nothing but an ordered groupoid.

\subsection{Constellations with range and ordered categories with restrictions}  \label{seclocr}

Let $\Q$ be a constellation. For all $s\in Q$, we define
$$s_D=\{e\in D(\Q)\mid s\cdot e\mbox{ exists, and for all }f\in D(\Q), \mbox{  if }s\cdot f\mbox{ exists then } e\leq f\}.$$

\begin{dfn}\label{def:range}
A {\em constellation with range} $\Q=(Q,\, \cdot\, , D, R)$ is a constellation $\Q$ in which for all $s\in \Q$ the set $s_D$ has a single element, namely $R(s)$, and, in addition, for all $s,t\in \Q$, if $s\cdot t$ exists then $R(s\cdot t)=R(R(s)\cdot t)$.
\end{dfn}  

The final part of Definition~\ref{def:range} is referred to as the {\em congruence condition}. Without this condition, $\Q$ would be a {\em constellation with pre-range}. 
Further, in this condition, if $s\cdot t$ exists, then so must $R(s)\cdot t$, in view of the following.

\begin{result} \cite{constgen}  \label{Rst}
Suppose that $\Q=(Q,\, \cdot\, ,D,R)$ is a constellation with pre-range. Let $s,t\in Q$. Then the following are equivalent:  $s\cdot t$ exists; $s\cdot D(t)$ exists; 
$R(s)\cdot D(t)$ exists; $R(s)\cdot t$ exists.  Further,   if $s\cdot t$ exists, then $R(s\cdot t)\leq R(t)$.
\end{result}  

The fact that $s_D$ is a singleton for every $s\in Q$ easily gives the following. 

\begin{result} \cite{constgen}  \label{Rlaws} A 
constellation with pre-range is normal.
\end{result}

The constellation $\C_X$ is a constellation with range in which $R(s)$ is the restriction of the identity map on $X$ to $\Img(s)$.  Every category is a constellation with range.

Morphisms between constellations with range are required to respect $R$.

\begin{dfn}
A {\em range radiant} is a radiant $\rho:\Q_1\rightarrow \Q_2$ between constellations with range that satisfies $R(s\rho)=R(s)\rho$ for all $s\in \Q_1$.  The definitions of strong radiants, isomorphisms and embeddings extend in the obvious ways to range radiants.  
\end{dfn} 

\begin{dfn}
An element $a$ of a constellation with range $\Q$ is {\em left cancellative} if whenever $a\cdot b=a\cdot c$, we have $R(a)\cdot b=R(a)\cdot c$.  Moreover $\Q$ is {\em left cancellative} if every element is left cancellative.
\end{dfn}

The above definition coincides with the usual definition of  ``epimorphism" if $\Q$ arises from a category.  It is also reminiscent of the condition that
an element $a$ of a semigroup $S$ be $\mathcal{L}^*$-related to an idempotent $R(a)$. Here  $\mathcal{L}^*$ is the relation given by $a\,\mathcal{L}^*\, b$ if  for any $x,y\in S^1$, we have $ax=ay$ if and only if $bx=by$. The constellation with range $\C_X$ is left cancellative. 

The following results may be found in \cite{constrange}.

\begin{result}   \label{corresp}
If $\Q=(Q,\, \cdot\, ,D,R)$ is a $($left cancellative$)$ constellation with range, then the derived category $(Q,\, \circ\, ,D,R)$, where $s\circ t:=s\cdot t$ is defined if and only if $R(s)=D(t)$, is an ordered $($left cancellative$)$ category with restrictions in which $\leq$ is the natural order on $\Q$ as a normal constellation, and for $e\leq D(s)$, we have that $e|s=e\cdot s$.

Conversely, if $\C=(C,\, \circ\, ,D,R,\leq)$ is a $($left cancellative$)$ ordered category with restrictions, then setting $s\cdot t$ equal to $s\circ (R(s)|t)$ whenever $R(s)\leq D(t)$ makes $(C,\, \cdot\, ,D,R)$ into a $($left cancellative$)$ constellation with range, and the given partial order is nothing but the natural order on the constellation.  
\end{result} 

If $\Q=(Q,\cdot\, ,D,R)$ is a constellation with range and $\C=(C,\circ ,D,R,\leq)$ is an ordered category with restrictions, denote by ${\bf C}(\Q)$ the ordered category with restrictions $(Q,\circ,D,R,\leq)$ obtained as in Result \ref{corresp} from $\Q$, and denote by ${\bf Q}(\C)$ the constellation with range $(C,\circ,D,R)$ obtained as in that result from the ordered category with restrictions ${\mc C}$.  These two constructions are mutually inverse.

\begin{result}  \label{1to1}
Let ${\mc C}$ be an ordered category with restrictions and ${\mc Q}$ a constellation with range.  Then 
${\bf C}({\bf Q}({\mc C}))={\mc C}$ and ${\bf Q}({\bf C}({\mc Q}))={\mc Q}$.
\end{result}

The class of constellations with range is a category if arrows are taken to be range radiants, and the class of ordered categories with restrictions is itself a category in which the arrows are functors $\rho$ that are order-preserving, which means that $s\leq t$ implies $s\rho\leq t\rho$. 

\begin{result}  \label{equiv}
The category with objects ordered categories with restrictions and arrows being order preserving radiants is isomorphic to the category with objects  constellations with range and  arrows being range radiants.
\end{result}

\subsection{The D-inverse constellation case}

\begin{pro}  \label{isopro}
Every D-inverse constellation $\Q$ is a constellation with range in which $R(s)=s'\cdot s=D(s')$ for all $s\in \Q$, which is right cancellative,  and  left cancellative as a constellation with range.
\end{pro}
\begin{proof}  Let $\Q$ be a D-inverse constellation and define $R(s)$ as in the statement of the proposition, for each $s\in \Q$.  Now $R(s)=s'\cdot s=D(s')\in D(\Q)$ and $s\cdot R(s)=s\cdot (s'\cdot s)=(s\cdot s')\cdot s=D(s)\cdot s=s$.  Suppose that $s\cdot e$ exists for some $e\in D(S)$.  Then $s'\cdot(s\cdot e)=(s'\cdot s)\cdot e=R(s)\cdot e$ exists, and so $R(s)\leq e$.  Since $\Q$ is normal, $\leq$ is a partial order, and so $\Q$ is a constellation with pre-range.  

If $s\cdot t$ (equivalently, by Result \ref{Rst}, $R(s)\cdot t$) exists, then $s'\cdot (s\cdot t)=(s'\cdot s)\cdot t$ exists.  Hence by Result \ref{Rst}, 
\[R(s\cdot t)=R((s\cdot R(s))\cdot t)=R(s\cdot (R(s)\cdot t))\leq R(R(s)\cdot t)\]\[=R((s'\cdot s)\cdot t)=R(s'\cdot (s\cdot t))\leq R(s\cdot t).\]
By normality, all are equal and so in particular $R(s\cdot t)=R(R(s)\cdot t)$.  Hence $\Q$ satisfies the congruence condition and so is a constellation with range.

From Proposition~\ref{monic}, $\Q$ is right cancellative.   If $a,b,c\in \Q$ and $a\cdot b=a\cdot c$ then since $a'\cdot a=R(a)$ where $a'$ is the D-inverse of $a$, we have that 
\[R(a)\cdot b=(a'\cdot a)\cdot b= a'\cdot (a\cdot b)=a'\cdot (a\cdot c)= (a'\cdot a)\cdot c=R(a)\cdot c\]
where all products exist by (Const2).   So $\Q$ is left cancellative as a constellation with range.
\end{proof}

A D-inverse constellation that is a category is nothing but a groupoid, as is easily seen. 
 
\begin{dfn}
If $\Pcon=(P,\, \cdot\, ,D,R)$ is a constellation with range, we say it is {\em strongly right cancellative} if $(P,\, \cdot\, )$ is right cancellative and satisfies the condition that for all $e,f\in D(P)$ and $s\in P$, if $R(e\cdot s)=R(f\cdot s)$, then $e=f$.  
\end{dfn}

\begin{pro}  \label{isopro2}
When viewed as a constellation with range as in Proposition \ref{isopro}, the D-inverse constellation $\Q$ is strongly right cancellative.
\end{pro}
\begin{proof}
Suppose that $s\in \Q$ and $e,f\in D(\Q)$ are such that $R(e\cdot s)=R(f\cdot s)$. Then  $(e\cdot s)'\cdot (e\cdot s)=(f\cdot s)'\cdot (f\cdot s)$.  But $s\cdot s'=D(s)$ exists, and so $(e\cdot s)\cdot s'=e\cdot (s \cdot s')=e\cdot D(s)=e$;  similarly, $(f\cdot s)\cdot s'=f$.  Hence 
$(e\cdot s)'\cdot ((e\cdot s)\cdot s')=(e\cdot s)'\cdot e=(e\cdot s)'$ and similarly $(f\cdot s)'\cdot ((f\cdot s)\cdot s')=(f\cdot s)'$.  So 
\bea
(e\cdot s)'&=&(e\cdot s)'\cdot ((e\cdot s)\cdot s')\\
&=&((e\cdot s)'\cdot (e\cdot s))\cdot s'\\
&=&((f\cdot s)'\cdot (f\cdot s))\cdot s'\\
&=&(f\cdot s)'\cdot ((f\cdot s)\cdot s')\\
&=&(f\cdot s)'. 
\eea 
It follows that $e\cdot s=f\cdot s$ and, by the right cancellative property, $e=f$.  So  $\Q$ is strongly right cancellative as a constellation with range.
\end{proof}

\begin{eg}  \label{counter2}
Let $X=\{x,y,z\}$ and $P=\{f,i,g,a,b\}\subseteq C_X$, with $f=\{(x,x)\}$, $i=\{(x,x),(y,y)\}$, $g=\{(z,z)\}$, $a=\{(x,z)\}$ and $b=\{(x,z),(y,z)\}$.  It is easy to see that $\mathcal{P}=(P,\cdot,D,R)$ is a subconstellation of  $\C_X=(C_X,\, \cdot\, ,D,R)$ that is also closed under range and hence is a constellation with range itself, moreover one that satisfies the congruence condition since $\C_X$ does.  
Moreover if $s\cdot u=t\cdot u$ for any $s,t,u\in P$, if $u\in D(P)$ then obviously $s=t$, but if not then $u\in \{a,b\}$ and $s,t\in D(P)$ and so $s=D(s)=D(s\cdot u)=D(t\cdot u)=D(t)=t$.  
However, $R(f\cdot b)=R(a)=g=R(b)=R(i\cdot b)$, yet $f\neq i$.  So even for constellations with range satisfying the congruence condition, the second part of the definition of being strongly right cancellative is independent of the right cancellative law.\end{eg}

Every D-inverse constellation is D-regular as a constellation, as well as being a constellation with range.  Conversely, we have the following.  

\begin{thm}  \label{main}
Let $\Q$ be a D-regular constellation with range.  Then $\Q$ is a D-inverse constellation, and $R(s)=D(s')$  for all $s\in \Q$.
\end{thm}
\begin{proof}
For $s\in \Q$, let $t\in \Q$ be such that $s\cdot t=D(s)$, and let $t'=R(s)\cdot t$, which exists by Result \ref{Rst}.  Let $u$ be such that $t\cdot u=D(t)$.  Then there exists $s\cdot(t\cdot u)=s\cdot D(t)=s$, and we also have $s=(s\cdot t)\cdot u=D(s)\cdot u$.  Since $D(s)\cdot s$ exists, $(s\cdot t)\cdot s$ exists, hence so too does $R(s\cdot t)\cdot s$ by Result \ref{Rst}. By the congruence condition, $R(t')=R(R(s)\cdot t)=R(s\cdot t)$, so $R(t')\cdot s$ and hence $t'\cdot s$ exists, again by Result \ref{Rst}.  
Then
\bea
t'\cdot s&=&(R(s)\cdot t)\cdot (D(s)\cdot u)\\
&=&((R(s)\cdot t)\cdot D(s))\cdot u\\
&=&(R(s)\cdot t)\cdot u\\
&=&R(s)\cdot (t\cdot u)\mbox{ since $t\cdot u$ exists}\\
&=&R(s)\cdot D(t)\\
&=&R(s).
\eea
We have that $s\cdot t'=s\cdot(R(s)\cdot t)=(s\cdot R(s))\cdot t=s\cdot t=D(s)$ and, as just shown, $t'\cdot s=R(s)$, so $D(t')=D(t'\cdot s)=R(s)$.  Hence $t'$ is a D-inverse of $s$, so $\Q$ is D-regular.  Moreover, $\Q$ is normal by Result \ref{Rlaws}, and so by Corollary \ref{corDinv}, $\Q$ is D-inverse.
\end{proof} 

We have that $\C_X$ is a constellation with range, but it is not D-regular (since it is not even right cancellative as a constellation, as easy examples show).  However, it is not hard to check that for all $a\in C_X$ there exists $b\in C_X$ for which $b\cdot a=R(a)$.  

Because they are constellations with range, D-inverse constellations will correspond to some types of ordered categories with restrictions by Result~\ref{1to1}. First we need to clarify the notion of radiant specific to D-inverse constellations; in fact from the discussion in Subsection~\ref{sub:dinv} we may simply take it to be a  radiant. (An analogous result is familiar for inverse semigroups.)

\begin{thm}  \label{invcorresp} The category with objects D-inverse constellations and  radiants as arrows is isomorphic to the category with objects ordered groupoids and  order-preserving functors as arrows, via the restriction of the functors as in Result \ref{equiv}.
\end{thm}
\begin{proof}
Suppose that $\Q$ is a D-inverse constellation.  Then viewing it as a constellation with range, ${\bf C}(\Q)$ is an ordered category with restrictions.  Now $D(s)=s\cdot s'$ and $D(s')=s'\cdot s$ in $\Q$, where $s'$ is the D-inverse of $s$ in $\Q$.  But $R(s)=D(s')$ and $R(s')=D(s^{\prime\prime})=D(s)$ by uniqueness, so $s\cdot s'=s\circ s'$ and $s'\cdot s=s'\circ s$ in ${\bf C}(\Q)$.  

It remains to show that (OG1) holds.  Suppose that $s,t\in {\bf C}(\Q)$, with $s',t'$ being the D-inverses of $s,t$ in $\Q$.  Suppose that
 $s\leq t$.  Then $s=D(s)|t$ with $D(s)\leq D(t)$, so $s=D(s)\cdot t=(s\cdot s')\cdot t$ in $\Q$. As $D(s)\leq D(t)$, so $(s\cdot s')\cdot (t\cdot t')$ exists and equals $s\cdot s'$. Then
$$s\cdot s'=(s\cdot s')\cdot (t\cdot t')=((s\cdot s')\cdot t)\cdot t'=s\cdot t',$$
which therefore exists. So $R(s)\cdot t'=D(s')\cdot t'$ exists and equals
$$(s'\cdot s)\cdot t'=s'\cdot (s\cdot t')=s'\cdot (s\cdot s')=(s'\cdot s)\cdot s'=D(s')\cdot s'=s',$$
so $s'\leq t'$.  We conclude that ${\bf C}(\Q)$ is an ordered groupoid.
 
For the converse, suppose that $\C$ is an ordered groupoid.  Then it is an ordered category with restrictions, so that ${\bf Q}(\C)$ is a constellation with range.  Moreover for $s\in {\bf Q}(\C)$, if $s'$ is its inverse in $C$, then $D(s)=R(s')$, so in $\C$ and hence in ${\bf Q}(\C)$, $s\cdot s'=s\circ s'=D(s)$ and $s'\cdot s=s'\circ s=R(s)=D(s')$, and so $s'$ is a D-inverse of $s$ in ${\bf Q}(\C)$, which is unique by normality (Proposition \ref{normunique}).  Hence ${\bf Q}(\C)$ is a D-inverse constellation.

The remainder of the details of the correspondence now follow from Result~\ref{equiv}, also using the comment before Proposition \ref{Dregim2}.
\end{proof}

In light of this result, Theorem \ref{normiso} follows from the analogous fact for ordered groupoids; see Theorem 9 in Section $4.1$ of \cite{lawsoninverse} for example.  Our proof is rather shorter than that of the corresponding result in \cite{lawsoninverse}.

Since inductive groupoids are precisely ordered groupoids in which the projections form a semilattice, the ESN theorem yields  the following immediate consequence of Theorem \ref{invcorresp}.

\begin{cor}  \label{corinv}
The category of D-inverse constellations in which $D(\Q)$ is a meet-semilattice under its natural order is isomorphic to the category of inverse semigroups.
\end{cor}

Starting with an inverse semigroup $S$, we may view it as a D-inverse constellation by setting $D(s)=ss'$ and $s\cdot t=st$ but only when $stt'=s$, as in Proposition \ref{inv2inv}, and then this may be viewed as an ordered groupoid as in Theorem \ref{invcorresp} by retaining $D(s)$, setting $R(s)=D(s')=s's$, setting $s\circ t=st$ whenever $R(s)=D(t)$, which is to say that $s's=tt'$, and the order is given by the natural order, namely $s\leq t$ if and only if $s=D(s)t=ss't$.  On the other hand, beginning with the same inverse semigroup $S$, the usual way to view it as an inductive (hence ordered) groupoid is to define $D(s)=ss'$, $R(s)=s's$, $s\circ t=st$ when $s's=R(s)=D(t)=tt'$, and the order given by $s\leq t$ if $s=ss't$.  Clearly, these two viewpoints coincide.

\section{Generalising constructions of Nambooripad and Lawson}  \label{sec:construct}

There are likely to be many settings in which established results obtained using ordered groupoids may be more elegantly obtained using D-inverse constellations.  It is therefore possible that D-inverse constellations will prove a useful tool in the study of both ordered groupoids and inverse semigroups.  They also provide a perspective through which results for ordered groupoids may be able to be generalised  to constellations.

The ESN Theorem itself was  generalised significantly (although with category isomorphism replaced by category equivalence) when Nambooripad showed in \cite{nambooripad} that the category of regular semigroups is equivalent to the category of ordered groupoids in which $D(\Q)$ is a regular biordered set in a way compatible with its structure as an ordered groupoid.  In light of Theorem \ref{invcorresp}, this fact could instead be stated in terms of D-inverse constellations rather than ordered groupoids.  

Suppose that $S$ is a regular semigroup.  As usual, let $V(a)\subseteq S$ denote the set of inverses of $a\in S$.  As part of his ``ESN theorem" for regular semigroups, Nambooripad showed that $P=\{(s,s')\mid s\in S,s'\in V(s)\}$ is an ordered groupoid, with $D((s,s'))=(ss',ss')$, $R((s,s'))=(s's,s's)$, and $(s,s')\circ (t,t')$ defined if and only if $s's=tt'$ and equal to $(st,t's')$, and $(s,s')\leq (t,t')$ means $s=ss't, s'=t'ss'$ and $ss'=ss'tt'=tt'ss'$.

Now consider an involuted semigroup $S$, meaning a semigroup equipped with an involution, that is, a unary operation $^*$ satisfying $(st)^*=t^*s^*$ and $s^{**}=s$ for all $s,t\in S$.  Let $E^*(S)=\{e\in S\mid e^*=e=e^2\}$, the set of {\em projections} in $S$.  Define $I^*(S)=\{s\in S\mid ss^*s=s\}$, the set of {\em partial isometries} of $S$.  In Section $4.2$ of \cite{lawsoninverse}, Lawson shows that $I^*(S)$ forms an ordered groupoid in which the identities are the projections, $D(s)=ss^*$, $R(s)=s^*s$ and $s\circ t=st$ is defined if and only if $R(s)=D(t)$. Further, $s\leq t$ means $s=D(s)t$ with $D(s)\leq D(t)$ under the usual ordering of idempotents in $E^*(S)=\{ss^*\mid s\in I(S)\}$ given by $e\leq f$ whenever $e=ef(=fe)$. For $e\in E^*(S)$ and $s\in I^*(S)$ with $e\leq D(s)$ we define the restriction $e|s=es$, and dually for corestrictions.  

In what follows, we give a common generalisation of these facts, but using D-inverse constellations rather than ordered groupoids.  The advantage of such an approach is that the order and restriction (and even range operation) need not be considered, only the constellation and domain (partial) operations. 

Let $S$ be a semigroup with $E$ a non-empty subset of $E(S)=\{e\in S\mid e^2=e\}$.  In \cite{GouldZenab}, an element $s\in S$ was said to be  {\em $E$-regular} if there exists $t\in S$ for which $sts=s$ with $st,ts\in E$; if $s$ is $E$-regular, it follows that there is $u\in S$ such that $sus=s$ and $usu=u$, with $su,us\in E$ (simply let $u=tst$).  Call such $u$ an {\em $E$-inverse of $s$}.  (Semigroups in which every element has an $E$-inverse for some $E\subseteq E(S)$ were studied in \cite{Dsemigen}.)  In the case of an involuted semigroup $S$, if $s\in I^*(S)$, then $s^*$ is an $E^*(S)$-inverse of $s$ (since from $ss^*s=s$, it follows that $s^*ss^*=s^*$, and $ss^*,s^*s\in E^*(S)$).
  
Let $R_E(S)$ be the set of $E$-regular elements of $S$, suppose that $T\subseteq R_E(S)$ with $E\subseteq T$, and let 
\[I_E(T)=\{(s,s')\mid s,s'\in T, s'\mbox{ is an $E$-inverse of }s\}.\] 
We define a partial binary operation and two unary operations on $I_E(T)$ as follows:
\bi
\item $(s,s')\cdot (t,t')=(st,t's')$ providing $s=stt'$ and $tt's'=s'$, and undefined otherwise; 
\item $D((s,s'))=(ss',ss')$, $R((s,s'))=(s's,s's)$;
\item $(s,s')'=(s',s)$.
\ei

We say $E$ is {\em $T$-normal} if it is such that, for all $e\in E$, $s\in T$ and $s'\in T$ an $E$-inverse of $s$, if $e=ess'=ss'e$ then $s'es\in E$.  

For example, let $T=S$ be a semigroup in which $E(S)\neq\emptyset$, and then $E=E(S)$ is trivially $T$-normal since $s'es\in E(S)$ if $s'$ is an inverse of $s$ and $e$ is an idempotent for which $ess'=e$.  This recovers Nambooripad's construction.  Lawson's example is the special case in which $S$ is an involuted semigroup, $E=E^*(S)$, and $T=I^*(S)$, the partial isometries of $S$: since $ss^*$ and $s^*s$ are projections, every partial isometry $s$ is $E^*(S)$-regular with inverse $s^*$.  A set of idempotents $E\subseteq E(S)$ for a semigroup $S$ is {\em reduced} if for any $e,f\in E$ we have $ef=f$ if and only if $fe=f$. It is easy to see that if $E$ is reduced, then $E$-inverses, if they exist, are unique. In the case at hand, $E^*(S)$ is reduced. To see this, notice that  for $e,f\in E^*(S)$, we have that $ef=f$ if and only if $fe=f^*e^*=(ef)^*=f^*=f$. Hence $s^*$ is the unique $E^*(S)$-inverse of $S$.  The set $E^*(S)$ is also $T$-normal, since if $s\in T$ and $e\in E^*(S)$ with $e=es^*s$, then $s^*es\in E^*(S)$, as is easily seen.  

\begin{thm}  \label{invconsteg}
Suppose that $S$ is a semigroup containing non-empty $E\subseteq E(S)$, with $E\subseteq T\subseteq R_E(S)$. Then $\mathcal{I}_E(T)=(I_E(T),\cdot,D)$ is a D-inverse constellation with $(s',s)$ the D-inverse of $(s,s')$ for all $(s,s')\in I_E(S)$, if and only if $E$ is $T$-normal.
\end{thm}
\begin{proof}
Suppose first that $E$ is $T$-normal. Pick $(s,s')\in I_E(T)$.  Since $ss',s's\in E$, $(s,s')'=(s',s)\in I_E(T)$ also, and $(e,e)\in I_E(T)$ for all $e\in E$ since $e$ is an $E$-inverse of itself, so in particular, $D((s,s'))\in I_E(T)$. 

Suppose that $(s,s')\in I_E(T)$ is a right identity element. Certainly  
$(s',s)\in I_E(T)$ and the product $(s',s)\cdot (s,s')$ exists. It follows
that $(s',s)=(s's, s's)$, so that $s=s's\in E$.

Conversely, if $e\in E$ then $(e,e)\in I_E(T)$, and if $(x,x')\cdot (e,e)=(xe,ex')$ exists for some $(x,x')\in I_E(T)$, then $x=xee=xe$ and $x'=eex'=ex'$, so $(x,x')\cdot (e,e)=(xe,ex')=(x,x')$.  Hence $(e,e)$ is a right identity. It follows that the set of right identities of $\mathcal{I}_E(T)$ is given by
 \[\mathcal{RI}(I_E(T))=\{(e,e)\mid e\in E\}=\{D((s,s'))\mid (s,s')\in I_E(T)\}.\]

Choose $(s,s')\in I_E(T)$.  Then $ss'\in E\subseteq I_E(T)$, $(ss',ss')\in I_E(T)$, and  $D((s,s'))\cdot (s,s')$ exists (since $(ss')(ss')=ss'$ in $S$) and equals $(s,s')$.  If $(s,s')=(e,e)\cdot (s,s')=(es,s'e)$ for some right identity $(e,e)$, then $ess'=e$, $ss'e=e$, $es=s$ and $s'e=s'$. Hence $ss'=(es)(s'e)=(ess')e=ee=e$, showing uniqueness.

Now pick $(s,s'),(t,t')\in I_E(T)$ for which $(s,s')\cdot (t,t')$ exists,
so that $s=stt'$ and $tt's'=s'$.  Then $(s,s')\cdot (t,t')=(st,t's')$ and we must show that $(st,t's')\in I_E(T)$. We have that $$(st)(t's')(st)=(stt')(s'st)=s(s'st)=st$$ and $$(t's')(st)(t's')=t's'(stt')s'=t's'ss'=t's',$$ with $(st)(t's')=ss'\in E$.  But also, $(t's')(st)=t'(s's)t\in E$ by the $T$-normality of $E$.  Hence $(s,s')\cdot (t,t')=(ss',t't)\in  I_E(T)$, so this partial operation is well-defined, and moreover $D((s,s')\cdot (t,t'))=(stt's',stt's')=(ss',ss')=D((s,s'))$ from above.

If $(s,s'),(t,t'),(u,u')\in I_E(T)$ with $(s,s')\cdot (t,t')=(st,t's')$ and $(t,t')\cdot (u,u')=(tu,u't')$ both existing, then $s=stt', tt's'=s', tuu'=t$ and $ uu't'=t'$, so $s(tuu't')=stt'=s$ and $(tuu't')s'=tt's'=s'$, giving that $(s,s')\cdot ((t,t')\cdot (u,u'))$ exists.    

If $(s,s')\cdot((t,t')\cdot (u,u'))$ exists then it equals $(s,s')\cdot(tu,u't')=(stu,u't's')$ and we have $t=tuu'$, $uu't'=t'$, $s=s(tu)(u't')=stt'$, and $s'=(tu)(u't')s'=tt's'$, so $(s,s')\cdot (t,t')=(st,t's)$ exists. Further,  $(st)uu'=st$ and $uu't's'=t's'$, and so $((s,s')\cdot(t,t'))\cdot (u,u'))=(st,t's')\cdot (u,u')$ exists and is equal to $(stu,u't's')=(s,s')\cdot((t,t')\cdot (u,u'))$.

Hence, $(I_E(T),\cdot,D)$ is a constellation, and $D(I_E(T))=\{(e,e)\mid e\in E\}$.

 It is easy to check normality. Finally, for all $s\in I_E(T)$, we have $(s,s')\cdot (s',s)$ exists (since $ss's=s$ and $s'ss'=s'$ in $S$) and equals $(ss',ss')=D((s,s'))$, and similarly $(s',s)\cdot (s,s')$ exists and equals $D((s',s))$.  We have completed the verification that $\mathcal{I}_E(T)$ is a D-inverse constellation.

Conversely, suppose that $\mathcal{I}_E(T)$ is a D-inverse constellation as in the statement of the theorem. Let $e\in E$ and suppose that $s'$ is an $E$-inverse of $s$ for which $e=ess'=ss'e$.  It follows that $(e,e)\cdot (s,s')$ exists and equals $(es,s'e)$, so $s'e$ is an $E$-inverse of $es$.  In particular, $s'es=s'e(es)\in E$.  Hence $E$ is $T$-normal.  
\end{proof}

As noted earlier, any semigroup $S$ is such that $E(S)$ is $S$-normal.  Consequently, we obtain the following.

\begin{cor}
Suppose that $S$ is a semigroup and let $R_{E(S)}(S)=\Reg(S)$, the set of regular elements of $S$.  Then  $(I_{E(S)}(\Reg(S)),\cdot\, ,D)$ is a $D$-inverse constellation.
\end{cor}

Consider the  case in which $S$ is regular,  that is, $\Reg(S)=S$. Upon translation into the language of ordered groupoids, we recover Nambooripad's construction of an ordered groupoid based on ordered pairs $(s,s')$ from a regular semigroup as in \cite{nambooripad}.  In fact the construction can take place within any semigroup, but only involves its regular elements.

As shown in \cite{Dsemigen}, $E$-inverses in the semigroup $S$ are unique when they exist if and only if $E$ is {\em pre-reduced}, meaning that for all $e,f\in E$, $ef=f$ and $fe=e$ imply $e=f$, and $ef=e$ and $fe=f$ imply $e=f$.  

\begin{cor}  \label{Tnorm}
Suppose that $S$ is a semigroup with $E\subseteq E(S)$, $T\subseteq R_E(S)$ with $E\subseteq T$.
If $E$ is pre-reduced and $T$-normal, then $T$ is a $D$-inverse constellation, where we define $s\cdot t=st$ if and only if $stt'=s$ and $tt's'=s'$, where $t'$ is the unique $E$-inverse of $t$ and similarly for $s$, and the D-inverse of $s$ is its $E$-inverse $s'$.
\end{cor}
\begin{proof}
The conditions on $S$ and $T$ ensure that any $s\in T$ has a unique $E$-inverse $s'\in T$.  Then, $(T,\cdot)\cong (I_E(T),\cdot)$ as partial algebras, via the isomorphism $s\leftrightarrow (s,s')$, where $s'$ is the unique $E$-inverse of $s\in S$, as is easily seen.  That $T$ is a D-inverse constellation now follows from Theorem \ref{invconsteg}.
\end{proof}

If $S$ is an involuted semigroup, recall that $I^*(S)$ is the set of partial isometries $T$ in $S$.  As noted in the discussion prior to Theorem \ref{invconsteg}, $E^*(S)$ is reduced, hence pre-reduced, and is $I^*(S)$-normal.  We can now apply Corollary \ref{Tnorm} to recover Theorem 3 in Section $4.2$ in \cite{lawsoninverse}: $I^*(S)$ is a D-inverse constellation, where we define $s\cdot t=st$ if and only if $stt^*=s$, hence it is an ordered groupoid in the way described in \cite{lawsoninverse}.  

Note that in the situation of Corollary \ref{Tnorm}, if $s\cdot t$ exists for some $s,t\in T$, then $st\in T$ has $E$-inverse $t's'\in T$; however, $t'\cdot s'$ may not be defined in the constellation $T$, so the law $(s\cdot t)'=t'\cdot s'$ fails in general (since the right hand side may not exist).

\section{Inverse constellations via pre-constellations}  \label{sec:pre}

The definition of inverse semigroups does not involve domain or range operations.  We conclude by giving a formulation of the D-inverse constellation concept that does not presuppose the existence of a domain operation.  

The following definition first appeared in \cite{demonic}, in a relation algebra setting.

\begin{dfn}
We say that $\Pcon=(P,\cdot)$ is a {\em pre-constellation} if $\cdot$ is a partial binary operation that satisfies $($Const1$)$ and $($Const2$)$ in the definition of a constellation.
\end{dfn}

Every semigroup is nothing but a pre-constellation in which all products are defined. 

In a constellation, $RI(\Pcon)=D(\Pcon)\subseteq E(\Pcon)$ as seen in Section~\ref{sec:back}. It follows that $D(\Pcon)$ is determined by the structure of $\Pcon$ as a constellation.  Hence, a pre-constellation is a constellation in at most one way.  In general, $D(\Pcon)$ and $E(\Pcon)$ can be different, as may be easily seen by considering $\C_X$.  However, by Proposition \ref{EPinRIP}, if $\Pcon$ is inverse then $RI(\Pcon)=D(\Pcon)=E(\Pcon)$.

We wish to define a notion of regularity in a pre-constellation.  Recall from Section~\ref{sec:Dreg} that a semigroup is regular if for all $a\in S$ there exists a $b\in S$ such that $a=aba$; in this case, by setting $c=bab$ (where $a=aba$), it follows that $a=aca$ {\em and} $c=cac$.

In the case of pre-constellations, there is ambiguity in the definition of $x\cdot y\cdot x$, so we must take a little care.

\begin{dfn}\label{defn:regp}
The pre-constellation $\Pcon$ is {\em regular} if for all $x\in P$ there is $y\in P$ such that $x=x\cdot(y\cdot x)$.
\end{dfn}

In Definition~\ref{defn:regp}, given $x=x\cdot(y\cdot x)$, we also have
$x=(x\cdot y)\cdot x$.  We then obtain a result generalising the  one for semigroups.

\begin{pro}
The pre-constellation $\Pcon$ is regular if and only if for all $x\in P$ there is $z\in P$ such that $x=x\cdot(z\cdot x)$ and $z=z\cdot (x\cdot z)$.
\end{pro}
\begin{proof}  This is a matter of patient calculation. For convenience, we provide the details.  Suppose that $\Pcon$ is regular.  For $x\in P$, choose $y\in P$ such that $x=x\cdot(y\cdot x)=(x\cdot y)\cdot x$, so $x\cdot y,y\cdot x$ both exist.  Hence so does $z=y\cdot(x\cdot y)$.  Moreover,
\bea
x\cdot(z\cdot x)&=&x\cdot ((y\cdot (x\cdot y))\cdot x)\\
&=&x\cdot ((y\cdot ((x\cdot y)\cdot x))\\
&=&x\cdot (y\cdot x)\\
&=&x,
\eea
and, continuing to  make use of (Const1) and (Const2),
\bea
z\cdot (x\cdot z)&=& (y\cdot (x\cdot y))\cdot (x\cdot (y\cdot (x\cdot y)))\\
&=&(y\cdot (x\cdot y))\cdot ((x\cdot y)\cdot (x\cdot y))\\
&=&(y\cdot (x\cdot y))\cdot (((x\cdot y)\cdot x)\cdot y)\\
&=&(y\cdot (x\cdot y))\cdot (x\cdot y)\\
&=&y\cdot ((x\cdot y)\cdot (x\cdot y))\\
&=&y\cdot (x\cdot y)\\
&=&z.
\eea

The converse is immediate.
\end{proof}

Next, we make the obvious definitions.

\begin{dfn}
If $\Pcon$ is a pre-constellation, we say that $t\in P$ is an {\em inverse} of $s\in P$ if $s\cdot(t\cdot s)=s$ and $t=t\cdot (s\cdot t)$, and $\Pcon$ is an {\em inverse pre-constellation} if every element of $\Pcon$ has a unique inverse.
\end{dfn}

So a pre-constellation $\Pcon$ is regular if and only if every element has an inverse, and inverse if every element has a unique inverse.  These definitions specialise back to give the usual semigroup definitions if $\Pcon$  is a semigroup.    

\begin{pro}  \label{regisinv}
Suppose that $\Pcon$  is a regular pre-constellation.  Then it is inverse if and only if for all $e,f\in E(\Pcon )$, if $e=e\cdot (f\cdot e)$ and $f=f\cdot (e\cdot f)$, then $e=f$.
\end{pro}
\begin{proof}
If $\Pcon$ is inverse, then the given condition on idempotents follows because $e$ is the unique inverse for itself.

Conversely, suppose that $\Pcon$ is regular and the stated condition on idempotents holds.  Suppose that $s\in P$ has both $t,u\in P$ as inverses.  So $s=s\cdot (t\cdot s)$, $t=t\cdot (s\cdot t)$, $s=s\cdot (u\cdot s)$ and $u=u\cdot (s\cdot u)$.  It follows that $s\cdot t, s\cdot u, t\cdot s$ and
$u\cdot s$ are all idempotent. One can then calculate in a standard way that
\[s\cdot u=(s\cdot u)\cdot ((s\cdot t)\cdot (s\cdot u))
\mbox{ and }s\cdot t=(s\cdot t)\cdot ((s\cdot u)\cdot (s\cdot t)),\]
so that we deduce $s\cdot t=s\cdot u$. Similarly, one obtains that
$u\cdot s=t\cdot s$. We then have 
 $$u=u\cdot (s\cdot u)=(u\cdot s)\cdot u=(t\cdot s)\cdot u=t\cdot (s\cdot u)=t\cdot (s\cdot t)=t,$$
giving that $\Pcon $ is inverse.
\end{proof}

If $\Pcon$ is a semigroup, the last result asserts that if $\Pcon$ is regular then it is inverse if and only if whenever $e=efe$ and $f=fef$ for two idempotents $e,f$, it must be that $e=f$,  a condition which is therefore equivalent to the more familiar commuting idempotents characterisation.

We now turn to the relationship between D-inverse constellations and inverse pre-constellations.  They are not the same thing, since inverse semigroups are examples of D-inverse constellations, but two idempotents $e,f$ in an inverse semigroup can be such that $ef\neq e$.  

\begin{pro}  \label{cond12}
Every D-inverse constellation $\Pcon=(P,\cdot\, ,D)$ has an inverse pre-constellation reduct $(P,\cdot)$ in which the inverses of elements are their D-inverses, satisfying the following:
\bi
\item for all $e,f\in E(\Pcon)$, if $e\cdot f$ and $f\cdot e$ exist, then they are equal;
\item if $s\cdot e$ exists for some $s\in P$ and $e\in E(\Pcon)$, then $s\cdot e=s$.
\ei
\end{pro}
\begin{proof}
Suppose that $\Pcon$ is a D-inverse constellation.  For all $s\in P$ we have  $s\cdot (s^{\prime}\cdot s)=s\cdot R(s)=s$, and similarly $s^{\prime}\cdot(s\cdot s^{\prime})=s^{\prime}$, so that $\Pcon$  is regular.

For $e,f\in E(\Pcon)$, from Proposition \ref{EPinRIP} we have that $e,f\in D(\Pcon)$, so if $e\cdot f$ and $f\cdot e$ exist, they are $e,f$ respectively, and hence equal since $\Pcon$ is normal.   Hence for all $e,f\in E(\Pcon)$, if $e=e\cdot (f\cdot e)$ and $f=f\cdot (e\cdot f)$, then in particular $e\cdot f$ and $f\cdot e$ exist, so that $e=f$.  It follows that $\Pcon$ is inverse by Proposition \ref{regisinv}, and so $s^{\prime}$ must be the unique inverse of $s$ for each $s\in P$.  The final point is clear from the definition of a constellation.
\end{proof}

Note that any inverse semigroup is an inverse pre-constellation satisfying the first condition but not the second,  showing that inverse pre-constellations are  more general than D-inverse constellations.
  
Next we show that the two conditions in Proposition \ref{cond12} in fact characterise (reducts of) D-inverse constellations amongst inverse pre-constellations.

\begin{thm}
If $\Pcon$ is a pre-constellation that is inverse and satisfies the two conditions in Proposition \ref{cond12}, then it is a D-inverse  constellation in which the inverse $s^{\prime}$ of $s\in P$ is its D-inverse $($and hence $D(s)=s\cdot s^{\prime}$ for all $s\in P$).
\end{thm}
\begin{proof}
Let $\Pcon$ be an inverse pre-constellation.  Define $D(s)=s\cdot s^{\prime}$ for all $s\in P$.  Let $E=\{s\cdot s^{\prime}\mid s\in P\}$.  We have remarked that $E\subseteq E(\Pcon)$.   Since $e^{\prime}=e$ for all $e\in E(\Pcon)$, we see that $E=E(\Pcon)$.  But $D(s)\cdot s=(s\cdot s^{\prime})\cdot s=s$, and if $e\cdot s=s$ for some $e\in E(P)$, then $(e\cdot s)\cdot s^{\prime}=s\cdot s^{\prime}$, so $e=e\cdot D(s)=D(s)$.  So $D(s)$ is the unique $f\in E(\Pcon)=D(\Pcon)$ such that $f\cdot s=s$. We are given that $E(\Pcon)\subseteq RI(\Pcon)$. If $s\in RI(\Pcon)$ then as $s=D(s)\cdot s$, we have that $s=D(s)$ so that $E(\Pcon)=D(\Pcon)=RI(\Pcon)$. Consequently,  $(P,\cdot\, ,D)$ is a constellation. Moreover, it is D-inverse since  $D(s)=s\cdot s^{\prime}$ for all $s\in P$, hence also $D(s^{\prime})=s^{\prime}\cdot s$ for all $s\in P$, since $s^{\prime\prime}=s$.
\end{proof}

\section*{Acknowledgement} The authors are grateful to a careful referee, whose remarks have helped improve our paper.

\end{document}